\documentclass{lms}

\newtheorem{theorem}{Theorem}[section] 
\newtheorem{lemma}[theorem]{Lemma}     

\newnumbered{assertion}{Assertion}    
\newnumbered{conjecture}{Conjecture}  
\newnumbered{definition}{Definition}
\newnumbered{hypothesis}{Hypothesis}
\newnumbered{remark}{Remark}
\newnumbered{note}{Note}
\newnumbered{observation}{Observation}
\newnumbered{problem}{Problem}
\newnumbered{question}{Question}
\newnumbered{algorithm}{Algorithm}
\newnumbered{example}{Example}
\newunnumbered{notation}{Notation} 

\author{Simeon Ball and Joaquim Monserrat}

\classno{51M04, 52C35}

 \usepackage{amsmath}
\usepackage{amssymb}
\usepackage{graphicx}
\usepackage{amscd}
\usepackage{graphicx} 
\usepackage{caption}

\title[Sylvester's problem in higher dimensions]{A generalisation of Sylvester's problem to higher dimensions}

\extraline{The first author acknowledges the support of the project MTM2014-54745-P of the Spanish {\em Ministerio de Econom\'ia y Competitividad.\\}

6 August 2016. }

\begin{document}

\baselineskip=17pt

\maketitle

\begin{abstract}
In this article we consider $S$ to be a set of points in $d$-space with the property that any $d$ points of $S$ span a hyperplane and not all the points of $S$ are contained in a hyperplane. The aim of this article is to introduce the function $e_d(n)$, which denotes the minimal number of hyperplanes meeting $S$ in precisely $d$ points, minimising over all such sets of points $S$ with $|S|=n$. 
\end{abstract}

\section{Introduction}

In 1893 Sylvester \cite{Sylvester1893} asked if it is possible to have a finite set of points $S$ in the plane, not all contained in a line,  with the property that no line contains precisely two points of $S$. Sylvester's problem was solved by Gallai \cite{Gallai1944} in 1944, who proved that there is always a line containing precisely two points of $S$. Since Gallai's proof, a number of articles (\cite{deBE1948}, \cite{CM1968}, \cite{CS1993}, \cite{Dirac1951},  \cite{GT2012}, \cite{KM1958}, \cite{Moser1957} for example) have been published that aim to determine the minimum number of lines $e_2(n)$ meeting $S$ in exactly two points, minimising over all sets of points $S$ with $|S|=n$, not all collinear.

\begin{figure}[!htp]
\begin{center}
\includegraphics[width=.34\textwidth]{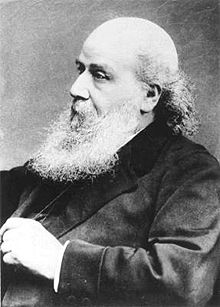}\\
\caption{J. J. Sylvester}
\end{center}
\end{figure}

A straightforward generalisation of Sylvester's problem to higher dimensions runs into difficulties. Motzkin \cite{Motzkin1951} observed that a finite set of points $S$ in $3$-space, distributed on two skew lines, has the property that no plane contains precisely three points of $S$. The survey article by Borwein and Moser \cite{BM1990}, and subsequently the book of problems by Brass, Moser and Pach \cite{BMP2005}, have a section on generalization of Sylvester's problem to higher dimensional spaces. The generalisation they consider is to minimise the number of hyperplanes $\pi$, with the property that all but one point of $\pi \cap S$ are contained in a hyperplane of $\pi$, again minimising over all sets of points $S$ with $|S|=n$.

In this note an alternative generalisation to higher dimensional spaces is proposed. Let $S$ be a set of points in $d$-space with the property that any $d$ points of $S$ span a hyperplane and not all the points of $S$ are contained in a hyperplane. Let $e_d(n)$ denote the minimal number of hyperplanes meeting $S$ in precisely $d$ points, minimising over all such sets of points $S$ with $|S|=n$. Note that for $d=2$ this coincides with the definition above since we automatically rule out double points in the planar case.

For any set  $S$ of points in $d$-space we say that a hyperplane $\pi$ is an {\em ordinary hyperplane} if $|\pi \cap S|=d$.

Throughout the article we shall consider $S$ to be a subset of points of PG$(d,\mathbb{R})$, the $d$-dimensional projective space over ${\mathbb R}$. This is no clearly no restriction if $S$ contains no points on the hyperplane at infinity $\pi_{\infty}$. If $S$ does contain points on $\pi_{\infty}$, we can apply a projective transformation which maps a hyperplane containing no points of $S$ to the hyperplane at infinity. In this way we obtain a set of points $S'$ which is contained in the affine part of the $d$-space, and with the same intersection properties with respect to hyperplanes as the original set $S$. Hence, it is also no restriction if $S$ does contain points of $\pi_{\infty}$. There are many texts providing background on projective spaces, see \cite{Coxeter2008} for example.

We will use the notation 
$$
\langle (x_0,\ldots,x_d) \rangle
$$
to denote a point of PG$(d,\mathbb{R})$, where $(x_0,\ldots,x_d)$ is a non-zero vector of ${\mathbb R}^{d+1}$.

\section{Preliminary results}

\begin{lemma} \label{projectS}
Let $S$ be a set of $n$ points of PG$(d,\mathbb{R})$ with the property that every $d$ points of $S$ span a hyperplane and $S$ is not contained in a hyperplane. For $x \in S$, denote by $S_x$ the set of $n-1$ points of PG$(d-1,\mathbb{R})$ obtained from $S$ by projecting from $x$. Then there is a point $x \in S$ for which
$$
dN \geqslant nN_x,
$$
where $N$ is the number of ordinary hyperplanes spanned by $S$ and $N_x$ is the number of ordinary hyperplanes spanned by $S_x$.
\end{lemma}

\begin{proof}
Counting in two ways the pairs $(x,\pi)$ where $x \in S$ and $\pi$ is an ordinary hyperplane of PG$(d,{\mathbb R})$, we have
$$
\sum_{x \in S} N_x= dN.
$$
The lemma follows from the pigeon-hole principle.
\end{proof}

For example, in Figure~\ref{cubeproject}, the eight point cube in PG$(3,\mathbb{R})$, which spans eight ordinary planes, is projected onto the seven point ``broken'' Fano plane in PG$(2,\mathbb{R})$ which spans three ordinary lines. Thus, in this particular case, we have equality in the inequality of Lemma~\ref{projectS}.

\begin{figure}[h]
\centering
\includegraphics[width=4.6 in]{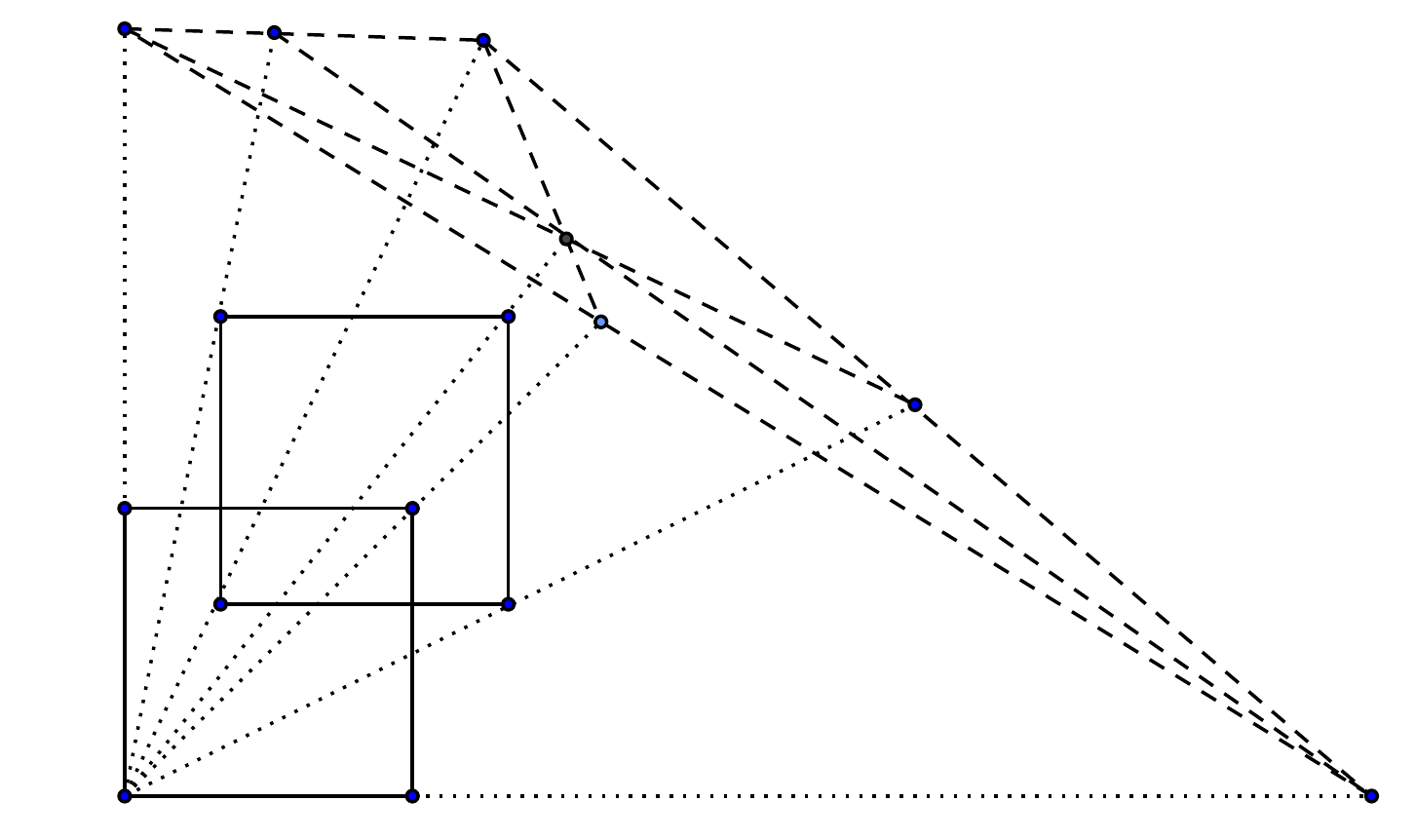}
\caption{The projection of a cube onto a ``broken'' Fano plane.}
\label{cubeproject}
\end{figure}

\begin{lemma} \label{project}
For $d \geqslant 3$,
$$
e_d(n) \geqslant \frac{n}{d}e_{d-1}(n-1).
$$
\end{lemma}

\begin{proof}
This follows immediately from Lemma~\ref{projectS}.
\end{proof}

\begin{lemma} \label{project2}
For $d \geqslant 3$,
$$
e_d(n) \geqslant \left\lceil \frac{n}{d} \left\lceil \frac{n-1}{d-1}\left\lceil  \frac{n-2}{d-2} \ldots \left\lceil  \frac{(n-d+3)}{3} e_2(n-d+2) \right\rceil \ldots \right\rceil \right\rceil \right\rceil .
$$
\end{lemma}

\begin{proof}
This follows by repeated application of Lemma~\ref{project} and the fact that $e_d(n)$ is an integer.
\end{proof}

Csima and Sawyer \cite{CS1993} proved that if $n \neq 7$ then $e_2(n) \geqslant 6n/13$, so we have the following theorem.

\begin{theorem}
For $n \neq d+5$,
$$
e_d(n) \geqslant \left\lceil \frac{n}{d} \left\lceil \frac{n-1}{d-1}\left\lceil  \frac{n-2}{d-2} \ldots \left\lceil  \frac{(n-d+3)}{3} \left\lceil \frac{6(n-d+2)}{13} \right\rceil \right\rceil \ldots \right\rceil \right\rceil \right\rceil .
$$
\end{theorem}

One of the main results of this article will be the following theorem which concerns the asymptotic behaviour of $e_d(n)$. We shall prove this theorem after we have deduced some structural theorem for sets of $n$ points which span few ordinary hyperplanes.

\begin{theorem} \label{upperlower}
For $n$ sufficiently large,
$$
e_2(n)=\left\{ \begin{array}{l} \frac{1}{2}n, \ \mathrm{if} \ n \ \mathrm{is\ even} \vspace{.2 cm}\\  \frac{3}{4}n-\frac{3}{4}, \  \mathrm{if} \ n \ \mathrm{is\ 1\ mod\ 4 } \vspace{.2 cm}\\  \frac{3}{4}n-\frac{9}{4}, \  \mathrm{if} \ n \ \mathrm{is\ 3\ mod\ 4 }\end{array}\right.
$$
$$
e_3(n)=\left\{ \begin{array}{l} \frac{1}{4}n^2-n, \ \mathrm{if} \ n \ \mathrm{is\ 0\ mod\ 4} \vspace{.2 cm}\\  \frac{3}{8}n^2-n+\frac{5}{8}, \  \mathrm{if} \ n \ \mathrm{is\ 1\ mod\ 4 } \vspace{.2 cm}\\  \frac{1}{4}n^2-\frac{	1}{2}n, \  \mathrm{if} \ n \ \mathrm{is\ 2\ mod\ 4 }\vspace{.2 cm}\\  \frac{3}{8}n^2-\frac{3}{2}n+\frac{17}{8}, \  \mathrm{if} \ n \ \mathrm{is\ 3\ mod\ 4 } \end{array}\right.
$$
and there is a universal constant $c$ for which
$$
\frac{3}{d!}n^{d-1}-\frac{c}{d!}n^{d-2} \leqslant e_d(n)\leqslant{n-1 \choose d-1}, \ \mathrm{if} \ d \geqslant 4.
$$
\end{theorem}

Let $S$ be a set of $n$ points of PG$(d,\mathbb{R})$ and let $\tau_i$ denote the number of hyperplanes containing $i$ points of $S$. We will call a hyperplane containing $i$ points of $S$, an {\em $i$-secant} hyperplane. Therefore, an ordinary hyperplane spanned by $S$ is a $d$-secant hyperplane and $\tau_d$ is the number of ordinary hyperplanes spanned by $S$.

The following is a simple counting lemma which we shall need.

\begin{lemma} \label{trivcount}
Let $S$ be a set of $n$ points of PG$(d,\mathbb{R})$ with the property that every $d$ points of $S$ span a hyperplane and $S$ is not contained in a hyperplane. Then
$$\sum_{i=d}^{n-1} {i \choose d} \tau_i= {n \choose d}.$$
\end{lemma} 

\begin{proof}
By counting $(d+1)$-tuples $(x_1,\ldots,x_d,\pi)$, where $x_1,\ldots,x_d \in S$ and $\pi$ is the hyperplane spanned by $x_1, \ldots, x_d$, in two ways.
\end{proof}

\section{Examples} \label{examples}

In the examples in this section we suppose that $n$ has at least $8$ points.

Let
$$
X_{2m}=\{ \langle (\cos(2\pi j/m), \sin (2 \pi j/m),1)\rangle \ | \ j=0,\ldots,m-1 \}
$$
$$
 \cup \{ \langle (-\sin (\pi j/m), \cos ( \pi j/m),0)\rangle  \ | \ j=0,\ldots,m-1\}.
$$

In Figure~\ref{affinebor}, the set $X_{12}$ is drawn in AG$(2,\mathbb{R})$, the line at infinity having been moved to the affine part, which accounts for the distortion of the regular polygon to six points on an ellipse.

\begin{figure}[h]
\centering
\includegraphics[width=6.1 in]{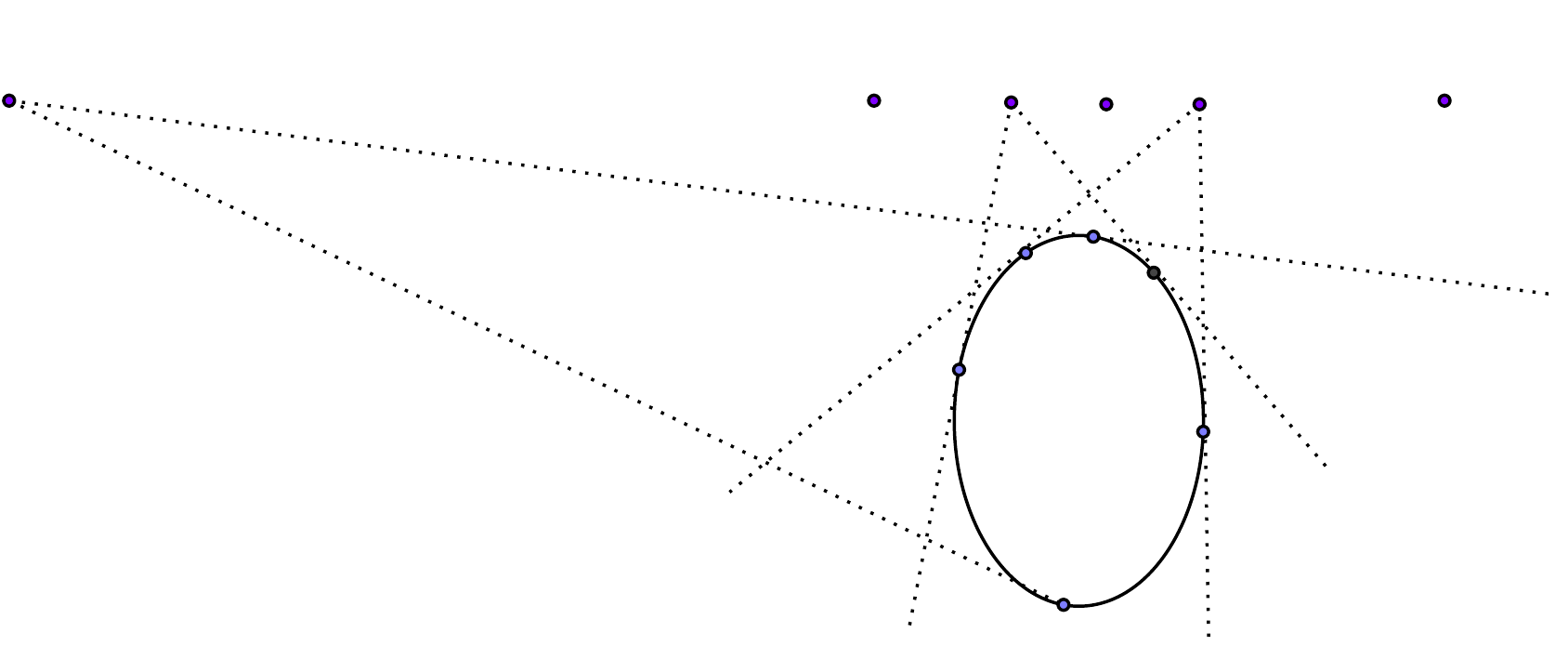}
\caption{The set $X_{12}$ spanning six ordinary lines.}
\label{affinebor}
\end{figure}

The following examples in PG$(2,\mathbb{R})$ were described by B\"or\"ocsky, as cited in \cite{CM1968}. The number of ordinary lines can be calculated using the sum of the angle formulas for the sine and co-sine functions, see \cite{GT2012}. In particular, one uses the fact that the line joining
$$
\langle (\cos(2\pi i/m), \sin (2 \pi i/m),1)\rangle \ \mathrm{and} \ \langle (\cos(2\pi j/m), \sin (2 \pi j/m),1)\rangle
$$
passes through the point
$$
\langle (-\sin (\pi (i+j)/m), \cos ( \pi (i+j)/m),0)\rangle.
$$

\begin{figure}[h]
\centering
\includegraphics[width=2.1 in]{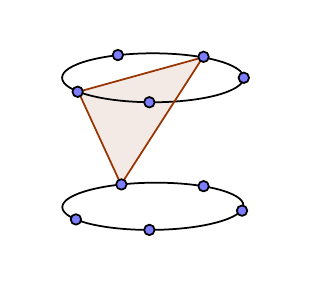}
\includegraphics[width=2.1 in]{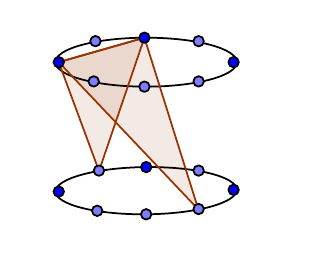}
\caption{The sets $P_{10}$ and $P_{16}$ spanning 20 and 48 ordinary planes respectively.}
\label{prism}
\end{figure}

\begin{lemma} \label{boro} {\rm (Regular polygon examples)}
If $n$ is even then the set $S=X_{n}$ spans $\frac{1}{2}n$ ordinary lines. If $n=1$ mod $4$ then the set $X_{n-1}$ together with the point $(0,0,1)$ spans $\frac{3}{4}n-\frac{3}{4}$ ordinary lines. If $n=3$ mod $4$ then the set $X_{n+1}$ with the point $(0,1,0)$ removed, spans $\frac{3}{4}n-\frac{9}{4}$ ordinary lines.
\end{lemma}

Let
$$
P_{2m}=\{ \langle (\cos(2\pi j/m), \sin (2 \pi j/m),1,0)\rangle \ | \ j=0,\ldots,m-1 \}
$$
$$
 \cup \{ \langle (\cos(2\pi j/m), \sin (2 \pi j/m),0,1) \ | \ j=0,\ldots,m-1\}.
$$

The following examples in PG$(3,\mathbb{R})$ were described by the first author in \cite{Ball2016}. The number of ordinary planes can be calculated again using the sum of the angle formulas for the sine and co-sine functions, see \cite{Ball2016}.

\begin{lemma} \label{prism} {\rm (The prism examples)}
If $n=0$ mod $4$ then the set $S=P_{n}$ spans $\frac{1}{4}n^2-n$ ordinary planes. If $n=2$ mod $4$ then the set $P_{n}$ spans $\frac{1}{4}n^2-\frac{1}{2}n$ ordinary planes.
If $n=1$ mod $4$ then the set $S=P_{n}$ with a point removed spans $\frac{3}{8}n^2-n+\frac{5}{8}$ ordinary planes. If $n=3$ mod $4$ then the set $P_{n}$ with a point removed spans $\frac{3}{8}n^2-\frac{3}{2}n+\frac{17}{8}$ ordinary planes. \end{lemma}

\begin{proof}
The plane
$$
\pi= \langle (\cos(2\pi i/m), \sin (2 \pi i/m),1,0), (\cos(2\pi j/m), \sin (2 \pi j/m),1,0), 
$$
$$
(\cos(2\pi k/m), \sin (2 \pi k/m),0,1)\rangle
 $$
contains the point $\langle (\cos(2\pi \ell/m), \sin (2 \pi \ell/m),0,1)\rangle$ if and only if there is an $\ell$ such that $i+j=k+\ell$. Therefore, $\pi$ is an ordinary plane spanned by $P_{2m}$ if $k$ satisfies $i+j=2k$ mod $m$.

If $m$ is odd then $i+j=2k$ mod $m$ has $\frac{1}{2}m(m-1)$ solutions where $i \neq j$. Therefore, if $n=2$ mod $4$ then a prism with $n$ points spans $\frac{1}{4}n^2-\frac{1}{2}n$ ordinary planes. If $m$ is even then $i+j=2k$ mod $m$ has $\frac{1}{2}m^2-m$ solutions where $i \neq j$. Therefore, if $n=0$ mod $4$ then a prism with $n$ points spans $\frac{1}{4}n^2-n$ ordinary planes.

By symmetry, every point of $S_n$ is incident with the same number of ordinary planes and $4$-secant planes. Therefore, by resolving the equation in Lemma~\ref{trivcount}, substituting $\tau_3$ and $\tau_{n/2}=2$, we can deduce $\tau_4$ and from that the precise number of $3$-secant planes and $4$-secant planes incident with a point of $S$. 

\end{proof}

The following example is the best known example for $d \geqslant 4$.

\begin{lemma} \label{trivex} {\rm (The trivial example, $n \geqslant d+2$)}
Let $S'$ be a set of $n-1$ points in a hyperplane $\pi$ with the property that every $d-1$ points of $S'$ span a hyperplane of $\pi$. Let $x$ be a point not in $\pi$ and let $S=S' \cup \{ x \}$. Then $S$ spans precisely ${n-1 \choose d-1}$ ordinary hyperplanes. For example, we could take
$$
S=\{ \langle (1,0,\ldots,0) \rangle \} \cup \{ \langle (0,1,t,t^2,\ldots,t^{d}) \rangle \ | \ t \in T \},
$$
where $T$ is a subset of ${\mathbb R}$ of size $n-1$.
\end{lemma}

\section{Structural theorems}

In \cite[Theorem 2.4]{GT2012} Green and Tao prove the following theorem.

\begin{theorem} \label{twodimstructure}
There is a constant $c$ such that  for $n$ sufficiently large, a set of $n$ points in PG$(2,{\mathbb R})$, spanning less than $n-c$ ordinary lines, is projectively equivalent to one of the regular polygon examples from Lemma~\ref{boro}. 
\end{theorem}

The following theorem for three-dimensional space is from \cite{Ball2016}. The same conclusion but with a slightly weaker bound of $\frac{1}{3}n^2-cn$ ordinary planes was obtained in \cite{Monserrat2015} for $n$ even.

\begin{theorem} \label{threedimstructure}
There is a constant $c$ such that for $n$ sufficiently large, a set of $n$ points in PG$(3,{\mathbb R})$, spanning less than $\frac{1}{2}n^2-cn$ ordinary planes, is projectively equivalent to either a prism, a skew-prism, a prism with a point deleted, a skew prism with a point deleted, or contains four collinear points.
\end{theorem}

The previous theorem has the following corollary. 

\begin{theorem} \label{ddimstructure}
There is a constant $c$ such that for $n$ sufficiently large, a set $S$ of $n$ points in PG$(d,{\mathbb R})$, $d \geqslant 4$, spanning less than $\frac{3}{d!}(n^{d-1}-cn^{d-2})$ ordinary hyperplanes, contains $d+1$ points that do not span a hyperplane.
\end{theorem}

\begin{proof}
If there are $d$ points of $S$ that do not span a hyperplane, then they span a smaller dimensional subspace. If this subspace contains a further point of $S$ then $S$ contains $d+1$ points that do not span a hyperplane, which is what we want to prove. If it doesn't contain a further point of $S$ then $S$ spans an infinite number of ordinary hyperplanes. Thus, we can assume that every subset of $d$ points of $S$ spans a hyperplane and will obtain a contradiction.

We consider the case $d=4$ first. 

Let $T$ be the subset of $S$ which consists of points which project to a set of $n-1$ points in $\mathrm{PG}(3,{\mathbb R})$ spanning less than $\frac{1}{2}n^2-c'n$ ordinary planes, for some constant $c'$. 

By Theorem~\ref{threedimstructure},  for any $x \in T$, the projection $S_x$, of $S$ from $x$, is contained in two planes. Therefore, there are two hyperplanes $\pi$ and $\pi'$ containing $x$ and all the points of $S$. Since there are $\frac{1}{2}(n-1)$ points of $S$ on each of these hyperplanes, these hyperplanes do not depend on $x$. Therefore $T \subseteq \pi \cap \pi'\cap S$, and since $|\pi \cap \pi'\cap S| \leqslant 3$, by the hypothesis on $S$, we have that $|T| \leqslant 3$. Hence, $S$ spans at least $\frac{1}{4}(n-3)(\frac{1}{2}n^2-c'n)$ ordinary hyperplanes so choosing $c$ large enough, we are done.

The theorem follows from Lemma~\ref{projectS} and the pigeon hole principle.
\end{proof}

We are now in a position to prove Theorem~\ref{upperlower}.

\begin{proof} (of Theorem~\ref{upperlower}.)

The asymptotic value of $e_2(n)$ follows from Lemma~\ref{boro} and Theorem~\ref{twodimstructure}.

The asymptotic value of $e_3(n)$ follows from Lemma~\ref{prism} and Theorem~\ref{threedimstructure}.

The asymptotic bounds on $e_d(n)$, $d \geqslant 4$ follow from Lemma~\ref{trivex} and Theorem~\ref{ddimstructure}.

\end{proof}

To complete this section we prove a specific structural theorem which we will require in the proof of Theorem~\ref{e3nine}.

\begin{theorem} \label{e28}
A set of eight points in PG$(2,{\mathbb R})$ spanning four ordinary lines is projectively equivalent to Example~\ref{boro}.
\end{theorem}

\begin{proof}
The proof is divided in two parts, in the first part we prove that $S$ must have a 4-secant, and in the second part we prove that $S$, up to a projective transformation, is the regular polygon example from Lemma~\ref{boro}.

Suppose that $\tau_4 = 0$. Since $\tau_2 = 4$, Lemma~\ref{trivcount} gives
$$
24 = 3\tau_3 + 10\tau_5 + 15\tau_6 + 21\tau_7.
$$
We have that $\tau_5 =\tau_6 =\tau_7 =0$, because if $\tau_7 =1$, then $\tau_3 =1$, but a configuration with one 7-secant and one 3-secant has at least 9 points. If $\tau_6 = 1$, then $\tau_3 = 3$, but a configuration with one 6-secant and three 3-secants has at least 10 points. If $\tau_5 \geqslant 1$, then the diophantine equation has no solution. So, $\tau_5 = \tau_6 = \tau_7 = 0$, and the equation implies $\tau_3 = 8.$

Any point of $S$ is incident with at most three 3-secants, since if not, then $S$ would have at least nine points.

Since there are eight 3-secants and eight points, and each point is incident with at most three 3-secants, each point is incident with exactly three 3-secants. There are four 2-secants, so each point of $S$ is incident with one 2-secant.  

Let us suppose that the points of $S$ are labelled $x_1$ to $x_8$ and that $\{x_{2j-1},x_{2j}\}$ is a 2-secant, for $j=1,2,3,4$.

Then, up to relabelling, there are four possible distrubtions for the pencils of 3-secants incident with the point $x_1$ and the point $x_2$, see Figure~\ref{e2eight}.

\begin{figure}[h]
\centering
\includegraphics[width=2.4 in]{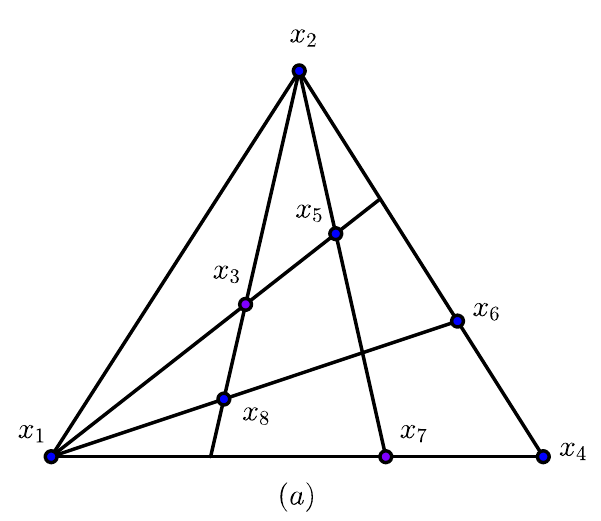}
\includegraphics[width=2.4 in]{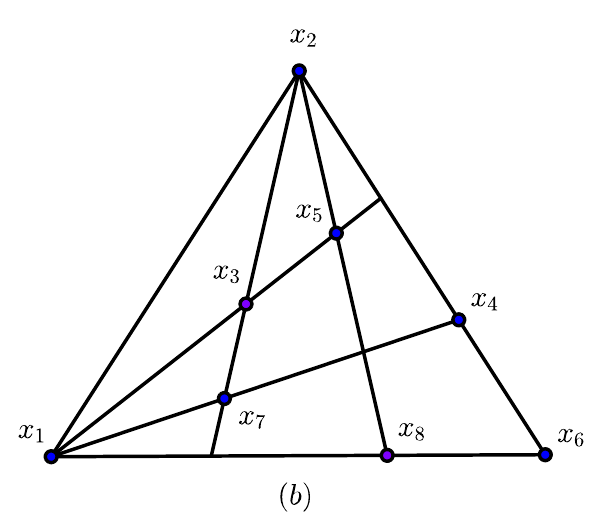}
\includegraphics[width=2.4 in]{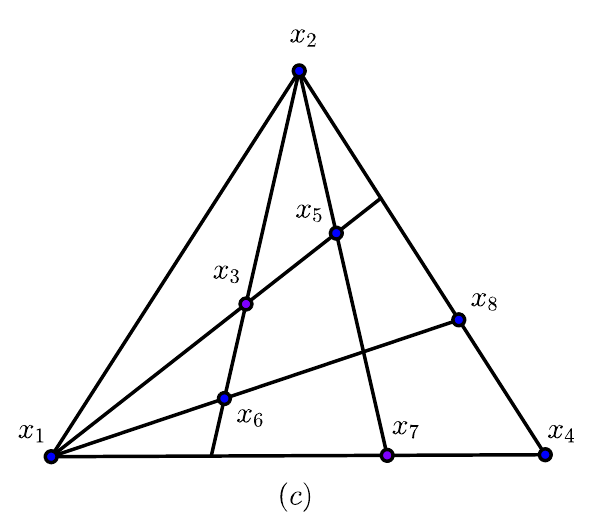}
\includegraphics[width=2.4 in]{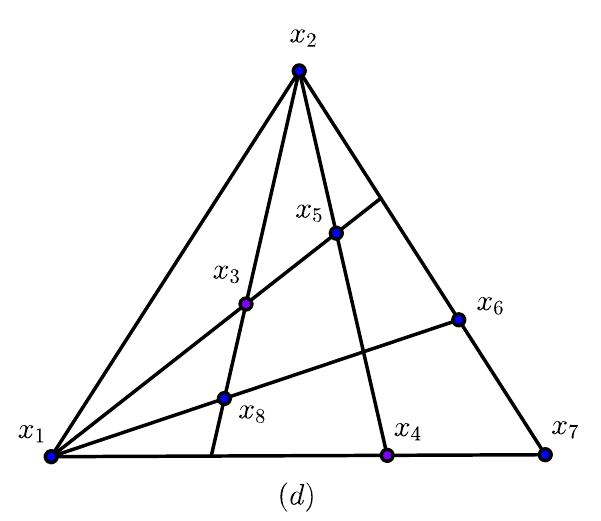}
\caption{The possible 3-secant distributions for lines incident with $x_1$ and $x_2$.}
\label{e2eight}
\end{figure}

The configuration of Figure~\ref{e2eight} $(b)$ is impossible, because the line pencil of $x_6$ would contain $\{x_3,x_6,x_7\}$, since $\{x_5,x_6\}$ is a $2$-secant, so there would be a 4-secant, $\{x_2, x_3,x_6,x_7\}$. The configuration of Figure~\ref{e2eight} $(c)$ is also impossible, since there would be a 4-secant, $\{x_1, x_4,x_6,x_7\}$. And the configuration of Figure~\ref{e2eight} $(d)$, because the line pencil of $x_6$ would be $\{x_5,x_6\}$, $\{x_1,x_6,x_8\}$, $\{x_2,x_6,x_7\}$, and $\{x_3,x_4,x_6\}$, contradicting the fact that $\{x_3,x_4\}$ is a 2-secant. So the only possible configuration is the one in Figure~\ref{e2eight} $(a)$, and has the lines incident with $\{x_1,x_2\}$, $\{x_3,x_4\}$, $\{x_5,x_6\}$, $\{x_7,x_8\}$, $\{x_1,x_3,x_5\}$, $\{x_1,x_8,x_6\}$, $\{x_1,x_7,x_4\}$, $\{x_2,x_3,x_8\}$, $\{x_2,x_5,x_7\}$, $\{x_2,x_6,x_4\}$, $\{x_3,x_6,x_7\}$ and $\{x_4,x_5,x_8\}$.

Now, we have to see that this configuration cannot be embedded in the plane. By the fundamental theorem of projective geometry, there is a projectivity that transforms the configuration of Figure~\ref{e2eight} to the points with coordinates
$$
x_1 = (0,0,1),\ x_2 = (0,1,0),\ x_3 = (1,0,0),\ x_4 = (1,1,1).
$$
Thus, by the relations of incidence we deduce that the other points have coordinates
$$
x_5 = (a,a,b),\ x_6 = (b-a, b, 0),\ x_7 = (a,0,b),\ x_8 = (b-a, b, b-a),
$$
where $a, b \in {\mathbb R} \setminus \{0\}$. The points $x_4$, $x_5$ and $x_8$ are collinear, so the determinant of these three points should be equal to zero,
$b^2-ab + a^2 = 0$. But this equation doesn't have a solution with $a, b \in {\mathbb R} \setminus \{0\}$.

Therefore, $S$ must have a $4$-secant, which by applying a suitable transformation we can assume is the line at infinity. The four affine points determine six (possibly repeated) directions and at least four distinct directions. A point of $S$ on the line at infinity is incident with $4$, $2$ or $0$ ordinary lines depending on whether it corresponds to direction determined by the affine points of $S$, zero, once or twice, respectively. Since $S$ spans only four ordinary lines, we have that the affine points of $S$ determine four distinct directions and the four affine points are affinely equivalent to the vertices of a square, so $S$ is projectively equivalent to the regular polygon example of Lemma~\ref{boro}.
\end{proof}

\section{The value of $e_d(n)$ for small $d$ and $n$}
%


\begin{lemma} \label{ints}
Let $S$ be a set of points of $\mathrm{PG}(d,{\mathbb R})$ with the property that every $d$-subset of $S$ spans a hyperplane. Let $T$ be a $(d+2)$-subset of $S$ spanning the whole space. There is at most one $(d+1)$-subset of $T$ which spans a hyperplane (the others span the whole space).
\end{lemma}

\begin{proof}
Suppose $Q_1$ and $Q_2$ are two $(d+1)$-subsets of $T$ which span distinct hyperplanes. Then $Q_1 \cap Q_2$ is a $d$-subset of $S$. Since $Q_1$ and $Q_2$ span distinct hyperplanes $Q_1 \cap Q_2$ does not span a hyperplane, contradicting the hypothesis on $S$.
\end{proof}

Recall that for a set of points $S$, we defined $\tau_i$ to be the number of hyperplanes containing $i$ points of $S$. 

\begin{lemma} \label{bettercount}
$$
 \sum_{i=1}^{n-d-1} (n-d-i) {d+i \choose i-1} \tau_{d+i} \leqslant {n \choose d+2}.
$$
\end{lemma}

\begin{proof}
Suppose $\pi$ is a $(d+i)$-secant, for some $i \geqslant 1$. There are ${d+i \choose d+1}$ subsets of $\pi \cap S$ of size $d+1$. For each of these $(d+1)$-subsets, if we add a point of $S\setminus \pi$ then, by Lemma~\ref{ints}, we obtain a distinct $(d+2)$-subset of $S$ spanning the whole space.
\end{proof}

The following theorem is useful only for $n\leqslant 2d$.

\begin{theorem} \label{lowerboundforsmalls}
$$
e_d(n) \geqslant {n \choose d} - \frac{d+1}{d+2}{n \choose d+1} .
$$
\end{theorem}

\begin{proof}
By Lemma~\ref{trivcount},
$$\sum_{i=0}^{n-d-1} {i+d \choose i} \tau_{i+d}= {n \choose d}.$$
So,
$$
\tau_d+\sum_{i=1}^{n-d-1} \frac{d+1}{i} {i+d \choose i-1} \tau_{i+d} = {n \choose d}.
$$
Since $(n-d-i)/(n-d-1) \geqslant 1/i$,
$$
\tau_d+\frac{d+1}{n-d-1}\sum_{i=1}^{n-d-1} (n-d-i)  {i+d \choose i-1} \tau_{i+d} \geqslant {n \choose d}.
$$
Now, use Lemma~\ref{bettercount}.
\end{proof}

\begin{theorem} \label{dplus2}
$$
e_d(d+2)={d+1 \choose 2}.
$$
\end{theorem}

\begin{proof}
This follows from Theorem~\ref{trivex}  and Theorem~\ref{lowerboundforsmalls}.
\end{proof}

\begin{theorem} \label{dplus3a}
If $d$ is odd then 
$$
e_d(d+3)=\tfrac{1}{6}(d+3)(d+1)(d-1).
$$
\end{theorem}

\begin{proof}
By Theorem~\ref{lowerboundforsmalls}, $e_d(d+3) \geqslant \tfrac{1}{6}(d+3)(d+1)(d-1)$.

We will construct a set $S$ of $d+3$ points with $\frac{1}{6}(d+3)(d+1)(d-1)$ hyperplanes containing precisely $d$ points of $S$.

Suppose $u_1,\ldots, u_{d+1}$ are $d+1$ points of PG$(d, {\mathbb R})$ which span PG$(d,{\mathbb R})$.

Let
$$
S=\{ u_1,\ldots, u_{d+1},u,v  \},
$$
where 
$$
u=u_1+\cdots+u_d,
$$ 
$$
v=\alpha_1 (u_1+u_2)+\cdots+\alpha_{(d-1)/2}(u_{d-2}+u_{d-1})+u_{d+1}
$$ 
and $\alpha_1,\ldots,\alpha_{(d-1)/2}$ are distinct elements of ${\mathbb R}$.

The hyperplanes $\langle u,u_1,\ldots, u_{d} \rangle$ and $\langle v,u_1,\ldots, u_{d-1}, u_{d+1} \rangle$ contain $d+1$ points of $S$. Furthermore,
$$
v-\alpha_1u=(\alpha_2-\alpha_1)(u_3+u_4)+\cdots+(\alpha_{(d-1)/2}-\alpha_1)(u_{d-2}+u_{d-1})+u_{d+1},
$$
so $\langle u,v,u_3,\ldots,u_{d+1} \rangle$ is also a hyperplane containing $d+1$ points of $S$. Similarly, by considering $v-\alpha_i u$ for $i=2,\ldots,(d-1)/2$, we find a further $(d-3)/2$ hyperplanes containing $d+1$ points of $S$. Hence
$$
\tau_{d+1} \geqslant (d+3)/2,
$$
and then Lemma~\ref{trivcount} gives
$$
\tau_d \leqslant  \tfrac{1}{6}(d+3)(d+1)(d-1).
$$
Hence, $\tau_d = \tfrac{1}{6}(d+3)(d+1)(d-1)$.
\end{proof}

\begin{theorem} \label{dplus3b}
If $d$ is even then 
$$
e_d(d+3)={d+2 \choose 3}.
$$
\end{theorem}

\begin{proof}
Lemma~\ref{bettercount} implies
$$
2\tau_{d+1}+(d+2)\tau_{d+2} \leqslant d+3.
$$
If $\tau_{d+2}=1$ then $\tau_{d+1}=0$ and Lemma~\ref{trivcount} implies 
$$
\tau_{d}={d+2 \choose 3}.
$$
If $\tau_{d+2}=0$ then since $d$ is even $\tau_{d+1} \leqslant (d+2)/2.$
Lemma~\ref{trivcount} implies
$$
\tau_d={d+3 \choose 3}-(d+1)\tau_{d+1}.
$$
Combining this with the above inequality gives 
$$
\tau_d \geqslant {d+2 \choose 3}.
$$
By Lemma~\ref{trivex}, 
$$
e_d(d+3) \leqslant {d+2 \choose 3}.
$$
\end{proof}

In the following table we list the values (or possible rangle of values) of $e_d(n)$, for small $n$ and $d$. The columns are indexed by $d$ and the rows by $n$. The column corresponding to $d=2$ comes from \cite{BM1990}. Any other entry which does not follow directly from Lemma~\ref{project}, Lemma~\ref{boro}, Lemma~\ref{trivex}, Theorem~\ref{dplus2}--~\ref{dplus3b} is justified below.

\bigskip



\begin{center}
\begin{tabular}{c|cccccc}
 & 2 & 3 & 4 & 5 & 6 & 7\\
\hline
4 & 3 & . & . & . & . & . \\
5 & 4 & 6 & . & . & . & . \\
6 & 3 & 8 & 10 & . & . & . \\
7 & 3 & 11 & 20 & 15 & . & . \\
8 & 4 & 8 & 25...35 & 32 & 21 & . \\
9 & 6 & 14..22 & 18...56 & 54...70 & 56 & 28 \\
10 & 5 & 20 & 35...84 & 36...126 & 90...126 & 80 \\
11 & 6 & 19...31 & 55...120 & 77...210 & . & .\\
12  & 6 & 24 & 57...165 & 132...330 & . & . \\
13 & 6 & 26...51 & 78...220 & 149...495 & . & . \\
\label{smallvalues}
\end{tabular}

The value of $e_d(n)$ for small $d$ and $n$.

\end{center}

%


\begin{theorem}
$$
e_3(7)=11.
$$
\end{theorem}

\begin{proof}
Consider four points of $S$ that span a plane $\pi$. The other three points of $S$ span a plane $\pi'$, which intersects $\pi$ in a line $\ell$. We consider the three possibilities for $\ell \cap S$ separately.

If $\ell \cap S= \emptyset$ then a hyperplane different from $\pi$ and containing at least four points of $S$ must contain two points $x,y \in S \cap \pi$ and two points $x',y' \in S \cap \pi'$ (so $\tau_5=\tau_6=0$). Let $T$ be the set of three points of $\ell$ which is the intersection of a line joining $x'$ and $y'$ (both points of $S \cap \pi'$) and $\ell$. The points of $T$ are on at most five lines joining two points of $S \cap \pi$, since otherwise the plane $\pi$ would contain a Fano plane. Thus, $\tau_4 \leqslant 5+1=6$, where the extra hyperplane containing four points of $S$ is $\pi$ itself. Lemma~\ref{trivcount} now implies $\tau_3 \geqslant 11$.

If $\ell \cap S=\{ z \}$ then a hyperplane different from $\pi$ and $\pi'$ and containing at least four points of $S 	\setminus \{ z \}$ must contain two points $x,y \in S \cap \pi$ and two points $x',y' \in S \cap \pi'$ (so $\tau_5=\tau_6=0$). This can be done in at most three ways. The point $z$ is on both $\pi$ and $\pi'$, which contain four points of $S$, but cannot belong to any further planes with four points of $S$, since such a plane would contain either two points of $\pi \setminus \{ z \}$ or two points of $\pi' \setminus \{ z \}$ and must therefore be either $\pi$ or $\pi'$. Thus, $\tau_4 \leqslant 5$ and Lemma~\ref{trivcount} now implies $\tau_3 \geqslant 15$.

If $\ell \cap S=\{ z,z' \}$ then either $\pi$ or $\pi'$ contains five points of $S$ and the other contains four points of $S$. Without loss of generality we can assume $\pi'$ contains five points of $S$.  The points $z$  and $z'$ cannot belong to any further planes with more than four points of $S$, since such a plane would contain either two points of $\pi \setminus \{ z, z' \}$ or two points of $\pi' \setminus \{ z , z'\}$ and must therefore be either $\pi$ or $\pi'$. The line joining the two points of $\pi \cap S  \setminus \{ z, z' \}$ can meet at most one line joining two points of $\pi' \cap S  \setminus \{ z, z' \}$ since $\pi' \cap S  \setminus \{ z, z' \}$  contains only three points. Therefore $\tau_4 \leqslant 1+1=2$. Lemma~\ref{trivcount} now implies $\tau_3 \geqslant 35-8-10=17$.

It only remains to provide an example. The cube with a vertex deleted has $\tau_4=6$ and $\tau_5=\tau_6=0$, so Lemma~\ref{trivcount} implies $\tau_3 =11$. Therefore, $e_3(7)=11$.
\end{proof}

\begin{theorem}
$$
e_4(8) \geqslant 25.
$$
\end{theorem}

\begin{proof}
Lemma~\ref{bettercount} implies 
$$
3\tau_5+12\tau_6+21 \tau_7 \leqslant 28
$$
and Lemma~\ref{trivcount} implies
$$
\tau_4 +5\tau_5+15\tau_6+35 \tau_7 =70.
$$
Hence
$$
3\tau_4 \geqslant 70+15 \tau_6.
$$
and so $\tau_4 \geqslant 24$. If $\tau_4=24$ then $\tau_6=0$ and so Lemma~\ref{trivcount} implies $\tau_5 \not\in {\mathbb Z}$. Hence $\tau_4 \geqslant 25$.
\end{proof}

In the same way one can show $e_5(9) \geqslant 54$.

\begin{theorem} \label{e3nine}
$$
e_3(9) \geqslant 14.
$$
\end{theorem}

\begin{proof}
We split the proof up depending on the number of $5$-secant planes.

Let $S$ be a set of $9$ points in $\mathrm{PG}(3,{\mathbb R})$ with the property that any  three points of $S$ span a plane.

Suppose $\tau_5 \geqslant 2$.  

Let $\pi$ and $\pi'$ be two $5$-secant planes and define $\ell=\pi \cap \pi'$. 

Since $|S|=9$ there is either one or two points of $S$ incident with $\ell$. 

If there is exactly one point $x \in S \cap \ell$ then $x$ projects $S$ onto a set of $8$ points in the plane such that the points are divided into two sets of four collinear points. Hence, $x$ is incident with $16$ ordinary planes and so $S$ spans at least $(16+8 \times 4)/3=16$ ordinary planes, since $e_2(8)=4$.

Suppose there are two points $x$ and $y \in S \cap \ell$. Each of these points projects $S$ onto a set of $8$ points, seven of which are contained in the union of two lines. The six points which are not the intersection of these lines span nine other lines of which at least six must be ordinary lines. Furthermore the projected point which is not on the union of the two lines, together with the point which is the intersection of the two lines, spans an ordinary line. Hence, both $x$ and $y$ are incident with at least $7$ ordinary planes. Therefore, $S$ spans at least $((2 \times 7)+(7 \times 4))/3=14$ ordinary planes.

Suppose $\tau_5=1$. By Theorem~\ref{e28}, a point not incident with the $5$-secant plane projects to a set of $8$ points in the plane spanning at least $5$ ordinary lines, so is incident with at least $5$ ordinary planes.. Therefore, $S$ spans at least $\lceil ((4 \times 5)+(5 \times 4))/3 \rceil=14$ ordinary planes. 

Suppose $\tau_5=0$. By Theorem~\ref{e28}, each point projects to a set of $8$ points in the plane spanning at least $5$ ordinary lines, so is incident with at least $5$ ordinary planes. Therefore, $S$ spans at least $(9 \times 5)/3 =15$ ordinary planes. 

\end{proof}

\section{Conclusions and conjectures}

This article introduces a problem in the hope that it will gain some attention. It seems to us a very interesting and natural question to ask and appears rather difficult to answer. In this final section we make some conjectures about the value of $e_d(n)$. Firstly, we recall the Dirac-Motzkin conjecture.

\begin{conjecture} \label{DM}
For all $n$, $e_2(n)\geqslant \lfloor \frac{1}{2}n \rfloor $.
\end{conjecture}

As we have seen, Green and Tao \cite{GT2012} proved Conjecture~\ref{DM} for $n\geqslant n_0$, where $n_0$ is large. Indeed, they prove more, that the exact values of $e_2(n)$, for $n \geqslant n_0$ are as in Theorem~\ref{upperlower}. It may be that the correct conjecture is that the exact values of $e_2(n)$ are as in Theorem~\ref{upperlower}, for all $n \geqslant n_0$, where $n_0$ is substantially smaller.  Note that for $n=13$, we have that $e_2(13)=6$, whereas $\frac{3}{4}(n-1)=9$, so we must take $n_0 \geqslant 14$.

We conjecture that the following is true.

\begin{conjecture} \label{BM}
For all $n$, $e_3(n)\geqslant \frac{1}{4}n^2-n$.
\end{conjecture}

Again, the asymptotic results from \cite{Ball2016}, imply that the exact values of $e_3(n)$, for $n \geqslant n_0$, where $n_0$ is large, are as in Theorem~\ref{upperlower}. It may be that the correct conjecture is that the exact values of $e_3(n)$ are as in Theorem~\ref{upperlower}, for all $n \geqslant n_0$, where $n_0$ is substantially smaller.  

Finally, we conjecture the following.

\begin{conjecture}
Suppose $d \geqslant 4$. There is a constant $c_d$, such that for $n$ sufficiently large, $e_d(n) \geqslant \frac{1}{(d-1)!}n^{d-1}-c_dn^{d-2}$.
\end{conjecture}

It is even possible that $e_d(n)={n-1 \choose d-1}$, for $d \geqslant 4$ and $n$ sufficiently large.

\bigskip
{\small Simeon Ball and Joaquim Monserrat}  \\
{\small Departament de Matem\`atiques}, \\
{\small Universitat Polit\`ecnica de Catalunya, Jordi Girona 1-3},
{\small M\`odul C3, Campus Nord,}\\
{\small 08034 Barcelona, Spain} \\
{\small {\tt simeon@ma4.upc.edu}}

\end{document}